\documentclass[letterpaper, 10 pt, conference]{ieeeconf}  

\IEEEoverridecommandlockouts                              
\usepackage{epsfig} 
\usepackage{amsmath,amssymb,amsfonts}
\usepackage{color}
\usepackage{balance}
\usepackage{booktabs,tabularx}
\usepackage{multirow}
\usepackage{mathtools}
\usepackage{pdfrender}
\usepackage{trfsigns}

\usepackage{centernot}
\usepackage{color}
\usepackage[acronym]{glossaries}
\usepackage[ruled]{algorithm2e}
\usepackage{graphicx,subcaption}
\usepackage{tabu}
\usepackage{cite}
\usepackage[official]{eurosym}

\graphicspath{{figures/}}

\newcommand{\norm}[1]{\|#1\|}

\newcommand{\R}{\mathbb{R}}

\newcommand{\Z}{\mathbb{Z}}
\newcommand{\N}{\mathbb{N}}

\newcommand{\mc}[1]{\mathcal{#1}}

\newcommand{\bx}{\bs{x}}

\newcommand{\bs}[1]{\boldsymbol{#1}}
\newcommand{\bsone}{\boldsymbol{1}}

\newcommand{\col}{\mathrm{col}}

\newtheorem{theorem}{Theorem}

\newtheorem{definition}{Definition}
\newtheorem{proposition}{Proposition}

\newtheorem{assumption}{Assumption}

\usepackage{xcolor,calc}

\makeglossaries
\newacronym{NEP}{NEP}{ Nash equilibrium problem}
\newacronym{MI-NEP}{MI-NEP}{mixed-integer Nash equilibrium problem}
\newacronym{MI}{MI}{mixed-integer}
\newacronym{VI}{VI}{variational inequality}
\newacronym{ICRF}{ICRF}{integer-compatible regularization function}
\newacronym{BR}{BR}{best-response}

\newglossaryentry{MI-NE}
{
	name={MI-NE},
	description={mixed-integer Nash equilibrium},
	first={\glsentrydesc{MI-NE} (\glsentrytext{MI-NE})},
	plural={mixed-integer Nash equilibrium},
	descriptionplural={mixed-integer Nash equilibria},
	firstplural={\glsentrydescplural{MI-NE} (MI-NE)}
}

\title{\LARGE \bf
Proximal-like algorithms for equilibrium seeking\\ in mixed-integer Nash equilibrium problems
}

\author{Filippo Fabiani$^1$, Barbara Franci$^2$, Simone Sagratella$^3$, Martin Schmidt$^4$ and Mathias Staudigl$^2$
\thanks{$^1$ At the time of the work, he was with the Department of Engineering Science, University of Oxford, OX1 3PJ, United Kingdom.
	Currently, he is with the IMT School for Advanced Studies Lucca, Piazza San Francesco 19, 55100 Lucca, Italy ({\tt \footnotesize filippo.fabiani@imtlucca.it}).}
\thanks{$^2$ Department of Advanced Computing Sciences, Maastricht University, NL--6200 MD, Maastricht, The Netherlands {\tt \footnotesize (\{b.franci, m.staudigl\}@maastrichtuniversity.nl)}.}
\thanks{$^3$ Department of Computer, Control, and Management Engineering  Antonio Ruberti (DIAG), Sapienza University of Rome, 00185, Rome, Italy {\tt \footnotesize (sagratella@diag.uniroma1.it)}. }
\thanks{$^4$ Department of Mathematics at Trier University, Universitätsring 15
54296 Trier, Germany {\tt\footnotesize(martin.schmidt@uni-trier.de)}. }
\thanks{This work was partially supported through the Government’s modern industrial strategy by Innovate UK under Project LEO (Ref. 104781) and the COST Action CA16228, ``European Network for Game Theory''}%
}

\begin{document}

\maketitle
\thispagestyle{empty}
\pagestyle{empty}

\begin{abstract}
  We consider potential games with mixed-integer variables, for which
  we propose two distributed, proximal-like equilibrium seeking
  algorithms. Specifically, we focus on two scenarios: i) the
  underlying game is generalized ordinal and the agents update through
  iterations by choosing an exact optimal strategy; ii) the game
  admits an exact potential and the agents adopt approximated optimal
  responses. By exploiting the properties of integer-compatible
  regularization functions used as penalty terms, 
 we show that both
  algorithms converge to either an exact or an $\epsilon$-approximate
  equilibrium. We corroborate our findings on a numerical instance of
  a Cournot oligopoly model.
\end{abstract}

\section{Introduction}\label{sec:intro}

Tracing back from the seminal work by Monderer and Shapley
\cite{monderer1996potential}, potential games represent a broad class
of noncooperative games characterized by the existence of a
real-valued function, the \emph{potential function}, such that any
collective strategy profile minimizing the underlying function
coincides with a Nash equilibrium of the game. Potential games hence
provide a means to naturally model many control-theoretic applications
\cite{Mar09} such as routing \cite{Orda93}, complex social networks
\cite{Staudigl:2011} and Cournot competition \cite{ConKluKraw04}.

We consider generalized ordinal potential games in
which part of the decision variables of the agents are constrained to
assume integer values.
Such \gls{MI} games have been recently proposed as strategic models for the distributed coordination of autonomous vehicles \cite{fabiani2018mixed,fabiani2019multi}, transportation and traffic control \cite{cenedese2021highway}, and smart grids \cite{Vujanic:2016wr,cenedese2019charging}.
In addition, \gls{MI} restrictions are often encountered in market games \cite{Gabriel:2013us,Sag17} and combinatorial congestion games \cite{Rosenthal:1973wg,KleSch21} as well.
For this practically relevant, yet intrinsically nonconvex, \gls{MI} game-theoretic setting the existence of equilibria follows by assuming that a certain master problem admits a solution, and we hence present two distributed, proximal-like equilibrium seeking algorithms. In particular, we consider the following scenarios: i) the underlying game is generalized ordinal and the agents update their control variables iteratively by choosing an exact proximal \gls{BR} strategy; ii) the game admits an exact potential function but we allow agents to choose an inexact proximal \gls{BR} for their updates.

Similar to Bregman-versions of proximal algorithms \cite{DvuShtStaSurvey21}, we formulate the proximal best-response function using a class of norm-like regularizers, known in the \gls{MI} optimization community as \glspl{ICRF}. We choose \glspl{ICRF} as penalty terms in place of standard quadratic regularizations as we believe they could provide us with a mean to include continuous reformulations of the \gls{MI} optimization subproblems. We leave this topic for future research.
Thus, by exploiting the properties of the \glspl{ICRF}, acting as penalty terms in the individual agent's \gls{BR} problems, we prove that both proposed algorithms enjoy convergence guarantees to an equilibrium of the \gls{MI-NEP}. Specifically, in the first scenario considered the computed \gls{MI-NE} is exact, while in the second one the algorithm returns an approximate \gls{MI-NE}.
Since \glspl{MI-NEP} constitute a rather new class of strategic optimization problems, there are not many solution techniques available.
To the best of our knowledge, this work represents a first attempt
proposing proximal-like distributed algorithms for a \gls{MI} game
setting. The only alternative applicable algorithm for \glspl{MI-NEP}
is the Gauss--Southwell method designed in
\cite{sagratella2017algorithms}. Given the practical relevance of
\glspl{MI-NEP}, this is rather surprising, and completely diametric to
continuous \glspl{NEP}, for which a whole
arsenal of numerical solution techniques is available
\cite{dreves2011solution,mertikopoulos2017convergence}. In fact,
proximal \gls{BR}-based algorithms have been extensively studied in
both stochastic and deterministic Nash games
\cite{Scu14,Lei:2019vj}. All these schemes, however, leverage the
\gls{VI} reformulation of \glspl{NEP} \cite{facchinei2007generalized},
and thus require strong (or strict) monotonicity of the \gls{VI},
assumptions that cannot be structurally satisfied in
\glspl{MI-NEP}. The algorithms we develop, in particular our adaptive
update of the penalty parameter regulating the proximal \gls{BR}, are
inspired by the decomposition method proposed in
\cite{facchinei2011decomposition}.
As main contributions we show i) how a proximal-like \gls{BR}-scheme
with \glspl{ICRF} can be used to compute Nash equilibria satisfying
\gls{MI} restrictions, and ii) we show convergence to an approximate
equilibrium even under inexact computations of \gls{BR} strategies.

The rest of the paper is organized as follows. In \S
\ref{sec:problem_formulation} we introduce the problem addressed,
along with some preliminaries, while in \S \ref{sec:algo} we
present the two methods and related convergence analysis. Finally,
in \S \ref{sec:simulations}, we test our findings on a \gls{MI}
Cournot oligopoly model and we conclude in \S \ref{sec:conclusion}.


\section{Problem formulation and preliminaries}\label{sec:problem_formulation}

Let $\mc{I} \coloneqq \{1, \ldots, N\}$ be the set indexing the
$N$~agents taking part in the noncooperative game $\Gamma \coloneqq
(\mc{I}, (J_i)_{i \in \mc{I}}, (\mc{X}_i)_{i \in \mc{I}})$. Each agent
controls \gls{MI} variables $x_i$ belonging to a compact, nonempty set
$\mc{X}_i \subseteq \R^{n^{c}_{i}} \times \Z^{n^{d}_{i}}$, and aims at
minimizing a given cost function $J_i : \mc{X} \to \R$ with $\mc{X}
\coloneqq \prod_{i \in \mc{I}} \mc{X}_i \subseteq \R^n$ and $n
\coloneqq \sum_{i \in \mc{I}} n_i = \sum_{i \in \mc{I}}(n^{c}_{i} +
n^{d}_{i})$. The resulting \gls{MI-NEP} thus reads
\begin{equation}\label{eq:single_prob}
  \forall i \in \mc{I} :
	 \underset{x_i \in \mc{X}_i}{\textrm{min}} \ J_i (x_i,  \bs{x}_{-i}) \,,
\end{equation}
where $\bs{x}_{-i} \coloneqq \textrm{col}((x_{j})_{j \in \mc{I}\setminus\{i\}})$. 
Given the strategies of the other agents, $\bs{x}_{-i}$, the \gls{MI}
\gls{BR} of agent $i\in\mc{I}$ is defined as
\begin{equation}\label{eq:br_mapping}
  B_{i}(\bs{x}_{-i}) \coloneqq \underset{x_i \in \mc{X}_i}{\textrm{argmin}} \ J_i (x_i,  \bs{x}_{-i}) \,.
\end{equation}
%
Our goal is to design distributed algorithms,
able to drive the set of agents to an \gls{MI-NE} of the game $\Gamma$, according to the definition given next.

\begin{definition}\textup{(Mixed-integer $\epsilon$-Nash equilibrium)}\label{def:MINE}
	Given some $\epsilon \geq 0$, a strategy profile $\bs{x}^{\ast} \in \mc{X}$ is an $\epsilon$-approximate \gls{MI-NE} (or $\epsilon$-\gls{MI-NE}) of the game $\Gamma$ if, for all $i \in \mc{I}$,
	\begin{equation}
		J_i(x^\ast_i, \bs{x}^\ast_{-i}) \leq J_i(y_i, \bs{x}^\ast_{-i})+\epsilon \text{ for all } y_i \in \mc{X}_i \,.
	\end{equation}
	%
	If $\epsilon=0$, then we call $\bs{x}^{\ast}$ an exact \gls{MI-NE}.
	\hfill$\square$
\end{definition}

Definition~\ref{def:MINE} points out that a \gls{MI-NE} of the game (if it exists) is achieved when all the agents adopt a \gls{BR} strategy.

\subsection{Generalized ordinal and exact potential games}
Existence theorems for \glspl{NEP} typically require continuity of the agents' cost functions as well as compactness and convexity of the feasible sets \cite{facchinei2007generalized}. 
Since the \gls{NEP} in \eqref{eq:single_prob} is nonconvex, existence of Nash equilibria is, in principle, not guaranteed. We thus focus on the classes of \emph{exact} and \emph{generalized ordinal potential games}, for which existence of solutions can be guaranteed under certain assumptions. 

\begin{definition}\textup{(Potential game)}\label{def:pot_game}
	A game $\Gamma$ is called
	\begin{itemize}
		\item \emph{exact} potential game \cite{monderer1996potential } if there exists a continuous function $P:\R^n \to \R$ such that, for all $i \in \mc{I}$,
		$$
			P(x_i, \bs{x}_{-i}) - P(y_i, \bs{x}_{-i}) = J_i(x_i, \bs{x}_{-i}) - J_i(y_i, \bs{x}_{-i}) \,,
		$$
		for all $\bs{x}_{-i}$ and $x_i$, $y_i \in \mc{X}_i$;
		\item \emph{generalized ordinal} potential game \cite{facchinei2011decomposition} if, there exists a \emph{forcing function} $\phi : \R_{\geq 0} \to \R_{\geq 0}$ such that, for all $i \in \mc{I}$, $\bs{x}_{-i}$, and $x_i, y_i \in \mc{X}_i$,
		$$
		\begin{aligned}
			&J_i(y_i, \bs{x}_{-i}) - J_i(x_i, \bs{x}_{-i}) > 0 \text{ implies}\\
			&P(y_i, \bs{x}_{-i}) - P(x_i, \bs{x}_{-i}) \geq \phi(J_i(y_i, \bs{x}_{-i}) - J_i(x_i, \bs{x}_{-i})) \, ,
		\end{aligned}
		$$
		where $\lim_{t\to 0^{+}} \phi(t) = 0$. \hfill$\square$
	\end{itemize}
\end{definition}

Note that any exact potential game is a generalized ordinal potential game.
%
By exploiting the tight relation between first-order information of the potential function and the local cost functions of the agents, it is well-known that potential functions can be employed in the construction of a suitable master problem facilitating the computation of equilibria.
\begin{proposition}\hspace{-.2em}\textup{\cite[Th.~2]{sagratella2017algorithms}}
Let $P$ be a generalized ordinal potential function for the game $\Gamma$. Given some $\epsilon\geq 0$, any $\epsilon$-approximate solution of the optimization problem
\begin{equation}\label{eq:Master}
	\underset{\bs{x}\in\mc{X}}{\textrm{min}}\ P(\bs{x})
\end{equation}
yields an $\epsilon$-approximate \gls{MI-NE} of $\Gamma$.
\hfill$\square$
\end{proposition}

\begin{assumption}\label{standing:existence}
  Problem \eqref{eq:Master} is solvable, i.e., there
  exists an $\bs{x}^{*}\in\mc{X}$ with $P(\bs{x})\geq P(\bs{x}^{*})$
  for all $\bs{x}\in\mc{X}$.
  \hfill$\square$
\end{assumption}

This assumption guarantees that the game~$\Gamma$ admits at least one Nash equilibrium in the nonconvex domain $\mc{X}$. Clearly, the solutions to the master problem \eqref{eq:Master} may not contain all possible Nash equilibria of the game $\Gamma$ \cite[Ex.~1]{sagratella2016computing}.

\subsection{Integer-compatible regularization functions}
Standard regularization techniques in \glspl{NEP} are based on the proximal \gls{BR} function obtained from \eqref{eq:br_mapping} by adding a quadratic penalty term.
Motivated by the proximal point interpretation of \gls{MI} optimization heuristics, we propose a regularization strategy of the individual agents' cost functions via \emph{integer-compatible regularization functions} (\glspl{ICRF}) \cite{boland2012new}. This family of functions has been introduced in the mixed-integer optimization community to control the duality gap \cite{Feizollahi2017}. They are also related to penalty methods for mixed-integer optimization \cite{lucidi2010exact}. 
\begin{definition}
\label{def:int_comp}
	A continuous function $\rho : \R^n \to \R$ is an \emph{integer-compatible regularization function} (\gls{ICRF}) if
	\begin{enumerate}
		\item[i)] $\rho(\bs{t}) \geq 0$ for all $\bs{t} \in \R^n$ and $\rho(\bs{t}) = 0 \iff \bs{t} = 0$;
		\item[ii)] for $\gamma \in (0,1)$, we have $\rho(\gamma \bs{t}) < \rho(\bs{t})$ for all $\bs{t} \neq 0$;
		\item[iii)] there exists a continuous and strictly increasing function $s : \R_{\geq 0} \to \R_{\geq 0}$ and some $\bar{K} \in \N$ such that, for all $K < \bar{K}$, $\rho(\bs{t}) \leq K \implies \|\bs{t}\|_1 \leq s^{-1}(K)$, where $\|\cdot\|_1$ denotes the $\ell_1$ norm in $\R^n$. \hfill$\square$
	\end{enumerate}
\end{definition}

Note that any norm defined in $\R^n$ is an \gls{ICRF}. A constructive
way to design \glspl{ICRF} is to consider decomposable penalties of
the form $\rho(\bs{t}) = \sum_{i=1}^{n} p(|t_i|)$, where $p : \R_{\geq
  0} \to \R_{\geq 0}$ is a concave and strictly increasing function. Prominent examples are
$p(t) = \log(t + \alpha) - \log(\alpha),\;p(t) = -(t + \alpha)^{-q} + \alpha^{-q},\;p(t)= 1 - e^{-\alpha t}$, or $p(t) = (1 + e^{-\alpha t})^{-1} - \tfrac{1}{2}$,
for some $\alpha$, $q > 0$ (see e.g. \cite{lucidi2010exact,boland2012new}). With these choices, $\rho(\cdot)$ amounts to an \gls{ICRF} \cite[Prop.~3.2]{boland2012new}. In particular, for the special case of binary constraints, a sensible formulation of an \gls{ICRF} is $\rho(\bs{t})=\sum_{i=1}^{n}\textrm{min}\{p(| t_{i}|),p(|1- t_{i}|)\}$.

\section{Proximal-like algorithms  for \glspl{MI-NEP}}\label{sec:algo}
We now propose two \gls{MI-NE} seeking algorithms. Algorithm \ref{alg:exact_ordinal} assumes that agents are able to compute an exact proximal \gls{BR} at each iteration. Algorithm \ref{alg:inexact_exact} relaxes this and instead allows for inexact BR computations. To show convergence, in the former case we rely on the fact that the \gls{MI-NEP} in \eqref{eq:single_prob} is generalized ordinal, whereas in the latter case we require the existence of an exact potential function.

Let $\rho_{i}$ denote the \gls{ICRF} employed by agent $i \in \mc{I}$
and let $\tau>0$ be a positive regularization parameter. We introduce
the \emph{proximal augmented local cost function} as a regularized
version of the local cost in \eqref{eq:single_prob}, which is given by
\begin{equation}\label{eq:aug_local}
	\tilde{J}_{i, \tau}(y_i, x_i; \bs{x}_{-i}) \coloneqq J_i(y_i, \bs{x}_{-i}) + \tau \rho_i(y_i - x_i)\,.
\end{equation}
In accordance, the \emph{proximal} \gls{BR} mapping in \eqref{eq:br_mapping} turns into
\begin{equation}\label{eq:perturbed_br_mapping}
	\beta_{i, \tau}(\bs{x}) \coloneqq \underset{y_i \in \mc{X}_i}{\textrm{argmin}} \ \tilde{J}_{i, \tau}(y_i, x_i; \bs{x}_{-i}) \,.
\end{equation}
Setting $\tau = 0$ allows us to recover the \gls{BR} mapping as
defined in \eqref{eq:br_mapping}, i.e., $\beta_{i, 0}(\bs{x}) =
B_i(\bs{x}_{-i})$. By considering these two new ingredients, in the
remainder of this section we design iterative and distributed schemes in
which the agents update their own action sequentially, according to
$\mc{I}$. We let $k\geq 0$ denote the iteration counter of the process
and $\bs{x}^{k}$ the iterate at the beginning of round $k+1$. For
an arbitrary agent $i\in\mc{I}$, we also define the (local) population state as
$$
\hat{\bs{x}}_i(k) \coloneqq \textrm{col}(x_1^{k+1}, \ldots,
x_{i-1}^{k+1}, x_i^{k}, x_{i+1}^{k},\ldots, x_N^{k}) \,.
$$
This corresponds to the collective vector of strategies at the $k$-th iteration communicated to agent $i$ when this agent has to perform an update, i.e., it computes the new strategy as a point in the \gls{MI} proximal \gls{BR} mapping,
$
x^{k+1}_{i}\in \beta_{i,\tau^{k}}(\hat{\bs{x}}_{i}(k)),
$
either exactly (Algorithm~\ref{alg:exact_ordinal}) or approximately
(Algorithm~\ref{alg:inexact_exact}). Successively, the next internal
state $\hat{\bs{x}}_{i+1}(k)$ is updated and passed to the ($i+1$)-th
agent. Throughout this process, note that, for all $k\geq 0$,
$\hat{\bs{x}}_{1}(k)=\bs{x}^{k}$ and $\hat{\bs{x}}_{N+1}(k)=\bs{x}^{k+1}$.


We stress that the proposed algorithms leverage the adaptive update of the regularization parameter $\tau$ in \eqref{eq:tau}, which produces a monotonically decreasing sequence $\{\tau^{k}\}_{k\geq 0}$, i.e., $\tau^{k+1}\leq \tau^{k}$ for all $k \geq 0$ \cite{facchinei2011decomposition}. Note that the rate of decrease strongly depends on $d_{\rho}(\bs{x}^{k+1},\bs{x}^{k})\coloneqq\textrm{max}_{i\in\mc{I}} \ \rho_{i}(x^{k+1}_{i}-x^{k}_{i}) \geq 0$, which measures the progress the method is making in the agents' proximal steps at the $k$-th iteration. As it will be clear from the convergence analysis, this quantity decreases over time, thus inducing a step towards an \gls{MI-NE}.

\subsection{Exact \gls{BR} computation in generalized ordinal potential games}
\label{subsec:exactcomp_genpot}
\begin{algorithm}[!t]
	\caption{Proximal-like method for \glspl{MI-NEP} with generalized ordinal potential and exact optimization}\label{alg:exact_ordinal}
	\DontPrintSemicolon
	\SetArgSty{}
	\SetKwFor{ForAll}{For all}{do}{End forall}\SetKwFor{While}{While}{do}{End while}
	\smallskip
	Set $k \coloneqq 0$, choose
        $\bs{x}^k\in \mc{X}$, $\tau^k > 0$ and $\omega \in (0,1)$\\
	\smallskip
	\While{$\bs{x}^k$ is not satisfying a stopping criterion}{
		\smallskip
		\ForAll {$i \in \mc{I}$}{
		\smallskip
		Obtain $\hat{\bs{x}}_{i}(k)$ and
		set  $x_i^{k+1} \in \beta_{i, \tau^k}(\hat{\bs{x}}_i(k))$\\
		\smallskip
		Update $\hat{\bs{x}}_{i+1}(k) = (x^{k+1}_{i}, \hat{\bs{x}}_{-i}(k))$
		}
		\smallskip
		Set $\bx^{k+1}=(x^{k+1}_{1},\ldots,x^{k+1}_{N})$\\
		\smallskip
		Update
		\begin{equation}\label{eq:tau}
                  \hspace{-0.15cm}\tau^{k+1}=\textrm{max}\left\{\omega\tau^{k},\textrm{min}\{\tau^{k},d_{\rho}(\bs{x}^{k+1},\bs{x}^{k})\}\right\}
		\end{equation}\\
		Set $k \coloneqq k+1$
		\smallskip
	}
\end{algorithm}

Next, by focusing on \glspl{MI-NEP} as in \eqref{eq:single_prob}, we  study the
convergence of Algorithm \ref{alg:exact_ordinal} under the following
assumption.

\begin{assumption}\label{ass:generalized_potential}
  The game $\Gamma = (\mc{I}, (J_i)_{i \in \mc{I}}, (\mc{X}_i)_{i \in
    \mc{I}})$ is a generalized ordinal potential game with potential
  function~$P(\cdot)$.
  \hfill$\square$
\end{assumption}

Thus, the Gauss--Seidel sequence of iterations in Algorithm~\ref{alg:exact_ordinal} has the following convergence property. We stress that, in our framework, an accumulation point for the sequence $\{\bs{x}^k\}_{k\geq 0}$ exists in view of Assumption~\ref{standing:existence}.
\begin{theorem}\label{th:exact_generalized}
	Under Assumption~\ref{ass:generalized_potential}, any accumulation point of the sequence of strategy profiles generated by Algorithm~\ref{alg:exact_ordinal}, $\{\bs{x}^k\}_{k\geq 0}$, is an \gls{MI-NE} of the game $\Gamma$ in \eqref{eq:single_prob}.
	\hfill$\square$
\end{theorem}
\begin{proof}
We first show that, in case the sequence $\{\bs{x}^k\}_{k\geq 0}$ generated by Algorithm~\ref{alg:exact_ordinal} admits a limit point $\bar{\bs{x}} \in \mc{X}$, then the regularization parameter, adaptively updated via \eqref{eq:tau}, satisfies $\lim_{k\to\infty}\tau^{k}=0$ and there exists an infinite index set $\mc{K}$ such that $\tau^{k+1}<\tau^{k}$ for all $k \in \mc{K}$. Then, we prove that $\bar{\bs{x}} \in \mc{X}$ is actually an \gls{MI-NE} of the \gls{MI-NEP} in \eqref{eq:single_prob}.

By construction of the regularization parameter sequence defined by \eqref{eq:tau}, we have $\tau^{k+1}\leq\tau^{k}$ for all $k \geq 0$. Then, for the sake of contradiction, assume that there exists some $\bar{\tau} > 0$ such that $\tau^{k}\geq \bar{\tau}$ for all $k \geq 0$. In view of the updating rule at the $k$-th iteration, for all $i \in \mc{I}$ we have
\begin{equation}
	\label{eq:LB}
	\begin{split}
		J_i(\hat{\bs{x}}_{i}(k))-J_i(\hat{\bs{x}}_{i+1}(k)) &\geq \tau^{k} \rho_{i}(x_{i}^{k+1}-x_{i}^{k})\\
		& \geq \bar{\tau}\rho_{i}(x^{k+1}_{i}-x^{k}_{i})\geq0\,.
	\end{split}
\end{equation}
By exploiting the definition of a generalized ordinal potential game provided in Definition~\ref{def:pot_game}, we hence deduce
\begin{equation}
  \label{LB2}
  \begin{split}
    & P(\hat{\bs{x}}_{i}(k))-P(\hat{\bs{x}}_{i+1}(k)) \\
    \geq \ & \phi(J_i(\hat{\bs{x}}_{i}(k))-J_i(\hat{\bs{x}}_{i+1}(k)))\geq 0\,.
  \end{split}
\end{equation}
Therefore,
$
	P(\bs{x}^{k+1})=P(\hat{\bs{x}}_{N+1}(k)) \leq P(\hat{\bs{x}}_{N}(k))
	\leq\dotsb\leq P(\hat{\bs{x}}_{1}(k))=P(\bs{x}^{k}),
$
and hence the sequence $\{P(\bs{x}^{k})\}_{k \geq 0}$ is monotonically non-increasing. By the continuity of $P$, it follows that the full sequence $\{P(\bs{x}^{k})\}_{k \geq 0}$ is convergent to a finite value $\bar{P}$. Moreover, it follows from \eqref{LB2} that
$
\lim_{k\to\infty} \phi(J_i(\hat{\bs{x}}_{i}(k))-J_i(\hat{\bs{x}}_{i+1}(k)))=0 \,.
$
By definition of the forcing function, we also have
$$
\lim _{k\to\infty}(J_i(\hat{\bs{x}}_{i}(k))-J_i(\hat{\bs{x}}_{i+1}(k)))=0 \,.
$$
From \eqref{eq:LB}, we obtain $\lim_{k\to\infty}\rho_{i}(x^{k+1}_{i}-x^{k}_{i})=0$. Let $r^{k} \to 0$ be such that $\rho_{i}(x^{k+1}_{i}-x^{k}_{i})\leq r^{k}$ for all $k$ sufficiently large. By definition of an \gls{ICRF}, we deduce $\norm{x^{k+1}_{i}-x^{k}_{i}}_{1} \leq s_{i}^{-1}(r^{k})$. We recall that the function $s_{i}^{-1}(\cdot)$ is monotonically increasing, and therefore
$
\lim_{k\to\infty} \norm{x^{k+1}_{i}-x^{k}_{i}}_{1}=0.
$
Consequently, for all $k$ sufficiently large, it holds that
$
d_{\rho}(\bs{x}^{k+1},\bs{x}^{k})<\bar{\tau}.
$
From \eqref{eq:tau}, we have $\tau^{k+1}=\textrm{max}\{\omega\tau^{k},d_{\rho}(\bs{x}^{k+1},\bs{x}^{k})\}$, and hence we need $\tau^{k+1}=\omega\tau^{k}$ for all $k$. However, this implies that $\tau^{k}\to 0$ at a geometric rate, thus contradicting the hypothesis that $\tau^{k}\geq\bar{\tau}>0$ for all $k \geq 0$ and concluding the first part of the proof.

Now, let $\{\bs{x}^{k}\}_{k \geq 0}$ be a convergent subsequence with
accumulation point $\bar{\bs{x}}\in\mc{X}$. The existence of such a
convergent subsequence is guaranteed by the compactness of
$\mc{X}$.
By invoking the same arguments as in the first part of the proof,
we obtain $\lim
_{k\to\infty}(J_i(\hat{\bs{x}}_{i}(k))-J_i(\hat{\bs{x}}_{i+1}(k)))=0$
and $\lim_{k\to\infty} \norm{x^{k+1}_{i}-x^{k}_{i}}_{1}=0$ for all
$i\in\mc{I}$. Next, we show by contradiction that the accumulation
point $\bar{\bs{x}} \in \mc{X}$ coincides with an \gls{MI-NE} of
\eqref{eq:single_prob} with generalized ordinal potential. To this
end, let us suppose that there exists an agent $i \in \mc{I}$ that can
further minimize its cost function, i.e., for some $y_{i} \in \mc{X}_{i}$,
$
J_i(y_{i}, \bar{\bs{x}}_{-i})<J_i(\bar{\bs{x}})
$
holds.
By relying on the update rule in Algorithm~\ref{alg:exact_ordinal}, we obtain
\begin{align*}
  & J_i(x_{i}^{k+1}, \bs{x}_{-i}^{k})+\tau^{k} \rho_{i}(x_{i}^{k+1}-x_{i}^{k})\\
  \leq \ & J_i(y_{i}, \bs{x}_{-i}^{k})+\tau^{k} \rho_{i}(y_{i}-x_{i}^{k})\,.
\end{align*}
Now, passing to the limit and using the fact that $\tau^{k} \rightarrow 0$,
	$$
	J_i(\bar{\bs{x}}) \leq J_i(y_{i}, \bar{\bs{x}}_{-i})<J_i(\bar{\bs{x}})\,,
	$$
	which denotes a contradiction, thus concluding the proof.
\end{proof}

In contrast to the continuous case where a constant penalty parameter
can be used (see, e.g., \cite[Th.~4.3]{facchinei2011decomposition}),
in a \gls{MI} setting a vanishing regularization parameter is
required. In addition, note that Algorithm~\ref{alg:exact_ordinal}
assumes that agents can compute an exact \gls{BR} of
the proximal augmented local cost function in
\eqref{eq:aug_local}. This requires that, at every single iteration,
each agent solves a \gls{MI} nonlinear optimization
problem to optimality. Without additional structural assumptions (e.g., individual convexity), this could render the method
inefficient in practice. This reason motivates us to
investigate a variant of Algorithm~\ref{alg:exact_ordinal} involving
inexact computation of a point lying in the perturbed \gls{MI}
\gls{BR} mapping \eqref{eq:perturbed_br_mapping} of each agent.

\subsection{Inexact \gls{BR} computation in exact potential games}
\label{subsec:exactpot_inexactcomp}

\begin{algorithm}[!t]
	\caption{Proximal-like method for \glspl{MI-NEP} with exact potential and inexact optimization}\label{alg:inexact_exact}
	\DontPrintSemicolon
	\SetArgSty{}
	\SetKwFor{ForAll}{For all}{do}{End forall}\SetKwFor{While}{While}{do}{End while}
	\SetKwFor{uIf}{If}{then}{}\SetKwFor{Else}{Else}{}{End}
	\smallskip
	Set $k \coloneqq 0$, choose $\bs{x}^k
        \in \bs{\mc{X}}$, $\tau^k > 0$, $\omega \in (0,1)$ and error
        tolerance sequence $\{\delta^k\}_{k\geq 0}$\\
	\smallskip
	\While{$\bs{x}^k$ is not satisfying a stopping criterion}{
		\smallskip
		\ForAll {$i \in \mc{I}$}{
			\smallskip
			Obtain $\hat{\bs{x}}_{i}(k)$ \\
			\smallskip
			\uIf{$x_i^{k} \in\hat{\beta}_{i,\tau^{k}}(\hat{\bs{x}}_{i}(k);\delta^{k})$}{
			\smallskip
			Set $x_i^{k+1} = x_i^{k}$ \\
			\smallskip
			}
			\Else{}{
			\smallskip
			Set $x_i^{k+1} \in\hat{\beta}_{i,\tau^{k}}(\hat{\bs{x}}_{i}(k);\delta^{k})$ \\
			\smallskip
			}
			\smallskip
			Update $\hat{\bs{x}}_{i+1}(k) = (x^{k+1}_{i}, \hat{\bs{x}}_{-i}(k))$\\
			\smallskip
		}
		\smallskip
		Set $\bx^{k+1}=(x^{k+1}_{1},\ldots,x^{k+1}_{N})$\\
		\smallskip
		Update $\tau^{k+1}$ as in \eqref{eq:tau}
		and set $k \coloneqq k+1$
		\smallskip
	}
\end{algorithm}

We now consider the case in which the \gls{MI-NEP} in \eqref{eq:single_prob} admits an exact potential function. In this case, we can prove convergence even if the agents implement only \emph{approximate} \gls{BR}s at every iteration, according to the following definition:

\begin{definition}($\delta$-proximal \gls{BR})
	Given any $\bs{x} \in \mc{X}$ and tolerance $\delta\geq 0$,
        $y_i \in \mc{X}_i$ is an $\delta$-optimal response to
        $\bs{x}_{-i}$ if
	$$\tilde{J}_{i, \tau}(y_i, x_i; \bs{x}_{-i}) \leq \tilde{J}_{i, \tau}(z_i, x_i; \bs{x}_{-i}) + \delta 
	$$
        for all $z_i \in \mc{X}_i$.\hfill$\square$
\end{definition}

For each agent $i \in \mc{I}$, we define
$
\hat{\beta}_{i, \tau}(\bs{x}; \delta) \coloneqq \{y_i \in \mc{X}_i \mid \tilde{J}_{i, \tau}(y_i, x_i; \bs{x}_{-i}) \leq \tilde{J}_{i, \tau}(z_i, x_i; \bs{x}_{-i}) + \delta \text{ for all } z_i \in \mc{X}_i\}
$
as the set of $\delta$-optimal responses, given some collective vector of strategies $\bs{x}$. For the theoretical developments of this subsection, we then make the following assumption.

\begin{assumption}\label{ass:exact_potential}
  The game $\Gamma = (\mc{I}, (J_i)_{i \in \mc{I}}, (\mc{X}_i)_{i \in
    \mc{I}})$ is an exact potential game with potential function
  $P(\cdot)$. \hfill$\square$
\end{assumption}

Algorithm~\ref{alg:inexact_exact} summarizes the main steps of the
resulting distributed, Gauss--Seidel type sequence of iterations. For
the considered instance, after choosing an initial strategy $\bs{x}^0
\in \mc{X}$ and a sequence of penalty parameters $\{\tau^k\}_{k\geq
  0}$, the preliminary step requires to further define a certain error
tolerance in computing a $\delta$-optimal response. To this end, we
let $\{\delta^{k}\}_{k\geq 0}$ be a given sequence of positive numbers
such that $\delta^{k}\to \epsilon$ for some $\epsilon \geq 0$. In Algorithm~\ref{alg:inexact_exact} we
stick to the sequential update architecture, but instead of requiring
that agents pursue an exact \gls{BR}, we allow the updating agent to
choose an inexact \gls{BR},
$x^{k+1}_{i}\in \hat{\beta}_{i,\tau^{k}}(\hat{\bs{x}}_{i}(k);\delta^{k})$, only in case the inexact \gls{BR} computed at the previous step, i.e., the one obtained by considering $\hat{\bs{x}}_{i}(k-1)$, $\tau^{k-1}$ and $\delta^{k-1}$, does not belong to $\hat{\beta}_{i,\tau^{k}}(\hat{\bs{x}}_{i}(k);\delta^{k})$. The
updated strategy is then sent to the agent down the line, and
the procedure repeats as long as some stopping criterion is not met.

\begin{theorem}\label{th:inexact_exact}
	Let Assumption~\ref{ass:exact_potential} holds true. Consider an error sequence $\{\delta^{k}\}_{k\geq 0}$ satisfying $\delta^k \to \epsilon$ with $\epsilon \geq 0$. Then any accumulation point of the sequence $\{\bs{x}^{k}\}_{k\geq 0}$ generated by Algorithm~\ref{alg:inexact_exact} is an $\epsilon$-\gls{MI-NE} of the game $\Gamma$ in \eqref{eq:single_prob}.
	\hfill$\square$
\end{theorem}

\begin{proof}
	The proof makes use of similar arguments as the one of Theorem~\ref{th:exact_generalized}. As a starting point, by definition of the update $x_i^{k+1}$, we have
	$$
	J_i(\hat{\bs{x}}_i(k))-J_i(\hat{\bs{x}}_{i+1}(k)) \geq \tau^{k} \rho_i(x_i^{k+1}-x_i^{k}) \,.
	$$
	Since in this case $P$ is an exact potential function, we have
	$$
	P(\hat{\bs{x}}_{i}(k))-P(\hat{\bs{x}}_{i+1}(k)) \geq \tau^{k} \rho_i(x_i^{k+1}-x_i^{k}) \geq 0 \,,
	$$
	where the last inequality uses the non-negativity of the \gls{ICRF}. By summing from $i=1$ to $N$, we hence obtain
	$$
	P(\bs{x}^{k})-P(\bs{x}^{k+1}) \geq 0 \,.
	$$
	In view of \cite[Lemma~3.4]{franci2022convergence}, it follows
        that
        $\lim_{k\to\infty}(P(\bs{x}^{k})-\textrm{min}_{\bs{x}\in\mc{X}}P(\bs{x}))$
        exists and is finite. Therefore, since $\mc{X} = \prod_{i \in
          \mc{I}} \mc{X}_i$ is compact, the sequence
        $\{\bs{x}^{k}\}_{k\geq 0}$ admits a convergent subsequence
        $\{\bs{x}^{k}\}_{k \in \mc{K}}$ with indices contained in some
        countable and infinite set $\mc{K}$, which has a limit point
        $\bar{\bs{x}} \in \mc{X}$. Therefore, in view of the
        continuity of the potential function, we have that $\lim _{k
          \rightarrow \infty}
        P(\bs{x}^{k})=P(\bar{\bs{x}})$ and $\lim _{k \rightarrow
          \infty}(P(\bs{x}^{k+1})-P(\bs{x}^{k}))=0$.
	Thus, for all $i \in \mc{I}$,
	$$
	\lim _{k \rightarrow \infty} (P(\hat{\bs{x}}_{i+1}(k))-P(\hat{\bs{x}}_i(k)))=0 \,.
	$$
	By definition of the potential function, it follows that
	$$
	\lim _{k \rightarrow \infty} (J_i(\hat{\bs{x}}_{i}(k))-J_i(\hat{\bs{x}}_{i+1}(k)))=0 \,.
	$$
	Again, by exploiting the property of the \gls{ICRF} in
        Definition~\ref{def:int_comp}.i), we obtain $\lim _{k
          \rightarrow \infty, k \in \mc{K}}\|x_i^{k+1}-x_i^{k}\|_1 =
        0$ for every $i \in \mc{I}$. It then follows
	$
	\lim _{k \rightarrow \infty} x_i^{k}=\bar{x}_i \text{ for all } i \in \mc{I}.
	$
Consequently, we deduce from the first part of the proof of Theorem \ref{th:exact_generalized} that also $\lim_{k\to\infty}\tau^{k}=0$.  We now claim that $\bar{\bs{x}} \in \mc{X}$ is an $\epsilon$-approximate Nash equilibrium and argue by contradiction. Suppose there exists an agent $i \in \mc{I}$ such that, for some $y_i \in \mc{X}_i$,
	$
	J_i(y_i, \bar{\bs{x}}_{-i}) + \epsilon <J_i(\bar{\bs{x}}).
	$
Resorting to the definition of the update mechanism yields
\begin{align*}
  & J_i(x_i^{k+1},\hat{\bs{x}}_{-i}(k)) + \tau^{k} \rho_i(x_i^{k+1}-x_i^{k})\\
  \leq \ & J_i(y_i,\hat{\bs{x}}_{-i}(k))+\tau^{k} \rho_i(y_i-x_i^{k})+\delta^{k} \,.
\end{align*}
	Then, passing to the limit and exploiting the fact that the regularization parameter tends to zero, i.e., $\lim _{k \rightarrow \infty} \tau^{k}=0$, and that $\delta^{k} \to \epsilon$, we finally obtain
	$$
	J_i(\bar{\bs{x}}) \leq J_i(y_i, \hat{\bs{x}}_{-i}) + \epsilon <J_i(\bar{\bs{x}}) \,.
	$$
	This represents a contradiction and concludes the proof.
\end{proof}


Note that Algorithm~\ref{alg:inexact_exact} also covers the case in
which the error sequence $\{\delta^k\}_{k \geq 0}$ is not forced by
the designer before implementing the procedure, but it naturally
arises from, e.g., a learning process. For example, consider a setting
in which the agents are endowed with typical learning procedures, such
as Gaussian processes or a neural network. As long as $\delta^{k} \to \epsilon$, Algorithm~\ref{alg:inexact_exact} returns an $\epsilon$-approximate \gls{MI-NE}. If the approximation error $\delta^{k}$ vanishes with the iteration index, instead, then Algorithm~\ref{alg:inexact_exact} produces a convergent sequence of strategy profiles to an exact \gls{MI-NE}. This assumption is not so stringent as it may seem. In the example considered above, i.e., agents endowed with learning procedures,  asymptotic consistency bounds on the approximation error can be exploited directly \cite{simonetto2021personalized}. 

\section{Numerical Experiments}\label{sec:simulations}
\begin{table}
  \caption{Simulation parameters}
  \label{tab:sim_val}
  \centering
  \begin{tabular}{llll}
    \toprule
    Symbol  & Unit & Description   & Value \\
    \midrule
    $p_i$ & \euro/good & Selling price & $\sim\mc{U}(10, 20)\times10^{3}$\\
    $m_i$ & \euro/good & Marginal cost & $\sim\mc{U}(7, 12)\times10^{3}$\\
    $u_i^d$ &  & Upper bound (discrete) & $\sim\mc{U}(200, 400)$\\
    $u_i^c$ &  & Upper bound (continuous)& $\sim\mc{U}(200, 400)$\\
    $\tau^0$ &  & Regularization parameter & $5000$\\
    $x_i^0$ &  & Initial goods vector & $\bs{0}_{100}$\\
    $\{\delta^k\}$ &  & Error tolerance sequence & $\tfrac{10^2 + (k^2 -1)\times10^{-6}}{k^2}$\\
    $\epsilon$ & & Approximation tolerance & $10^{-6}$\\
    \bottomrule
  \end{tabular}
\end{table}

We now test our theoretical findings on a numerical instance of a classic Cournot oligopoly model \cite{ConKluKraw04,Gabriel:2013us,sagratella2016computing}.

Specifically, we consider a market in which $N = 20$ firms produce
$n_{i} = 100$ goods each to maximize their
profits. Here, the first $50$ products, $x_i^d$, are
indivisible, while the other $50$, $x_i^c$, are modeled with
continuous variables and hence $x_i = \col(x_i^d, x_i^c)$. Thus, each
firm~$i \in \mc{I}$ aims at solving the following \gls{MI} quadratic
program
\begin{equation}\label{eq:cournot}
   \forall i \in \mc{I} : \left\{
  \begin{aligned}
    &\underset{x_i}{\textrm{min}} && (m_i - p_i)^\top x_i + (\textstyle\sum_{j \in \mc{I}} C_{i,j} x_j)^\top x_i\\
    &\textrm{ s.t. } && x^d_i \in \{0, u^d_i\}^{50}, \ x^c_i \in [0, u^c_i]^{50},
  \end{aligned}
\right.
\end{equation}
where the main parameters are described in
Table~\ref{tab:sim_val}. Also, we define matrices $C_{i,j} \in \R^{100
  \times 100}$, typically related with the costumers' inverse demand,
according to the procedure described in \cite[\S
4]{sagratella2016computing} for all $j \geq i$, and then we impose
$C_{j,i} = C^\top_{i,j}$. This degree of symmetry
ensures the existence of a potential function for the \gls{MI-NEP} in
\eqref{eq:cournot}, see, e.g., \cite{cenedese2019charging}, which has the form
$$
P(\bs{x}) = \textstyle\sum_{i \in \mc{I}} \left( h_i(x_i)  + \textstyle\sum_{j \in \mc{I}, j < i} g(x_i,x_j) \right)
$$
with $h_i(x_i) \coloneqq (m_i - p_i)^\top x_i + x_i^\top C_{i,i} x_i$ and $g(x_i,x_j) \coloneqq x_j^\top C_{j,i} x_i$.
In case the $(l,q)$-entry of matrix $C_{i,j}$ is nonnegative, the $q$-th product of the firm $j$ is a \emph{substitute} for the $l$-th product of firm $i$. On the other hand, if that entry is negative then the $q$-th product of the firm $j$ is a \emph{complement} for the $l$-th product of firm $i$. This framework resembles the 2-groups partitionable class in \cite{sagratella2016computing}. 
As an \gls{ICRF} for each agent, we adopted a piecewise affine approximation of the function $\rho_i(t_i) = \sum_{p=1}^{100} (1 - e^{0.9 |t_{i,p}|})$ applied componentwise ($t_{i,p}$ denotes the $p$-th element of $t_i \coloneqq y_i - x_i$). 
This is in order to handle each problem in \eqref{eq:cournot} with the solver used.

\begin{figure*}[t]
	\begin{subfigure}{.7\columnwidth}
		\centering
		\includegraphics[width=\textwidth]{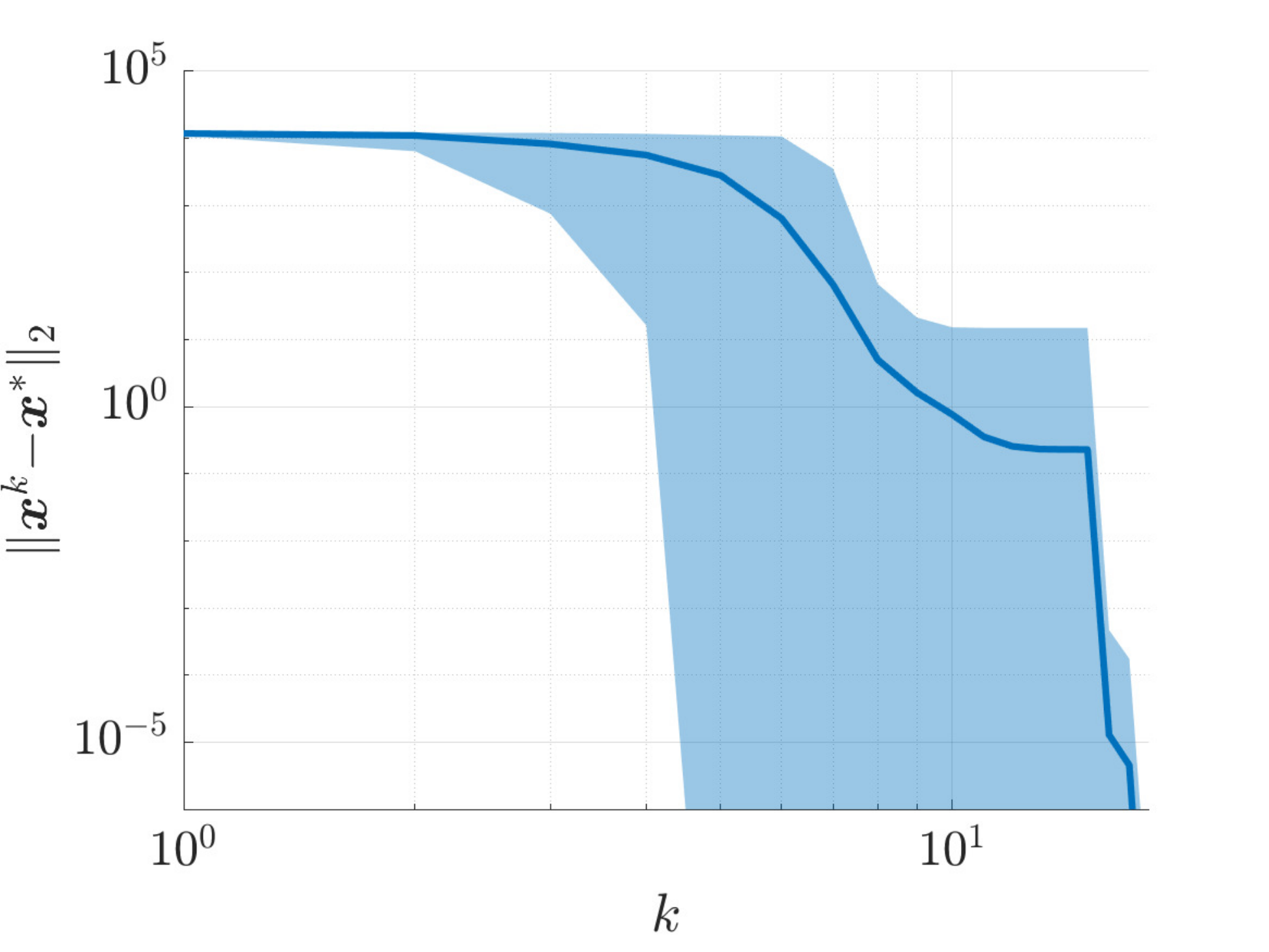}
		\caption{}
		\label{fig:MINE_distance}
	\end{subfigure}
	\begin{subfigure}{.7\columnwidth}
		\centering
		\includegraphics[width=\textwidth]{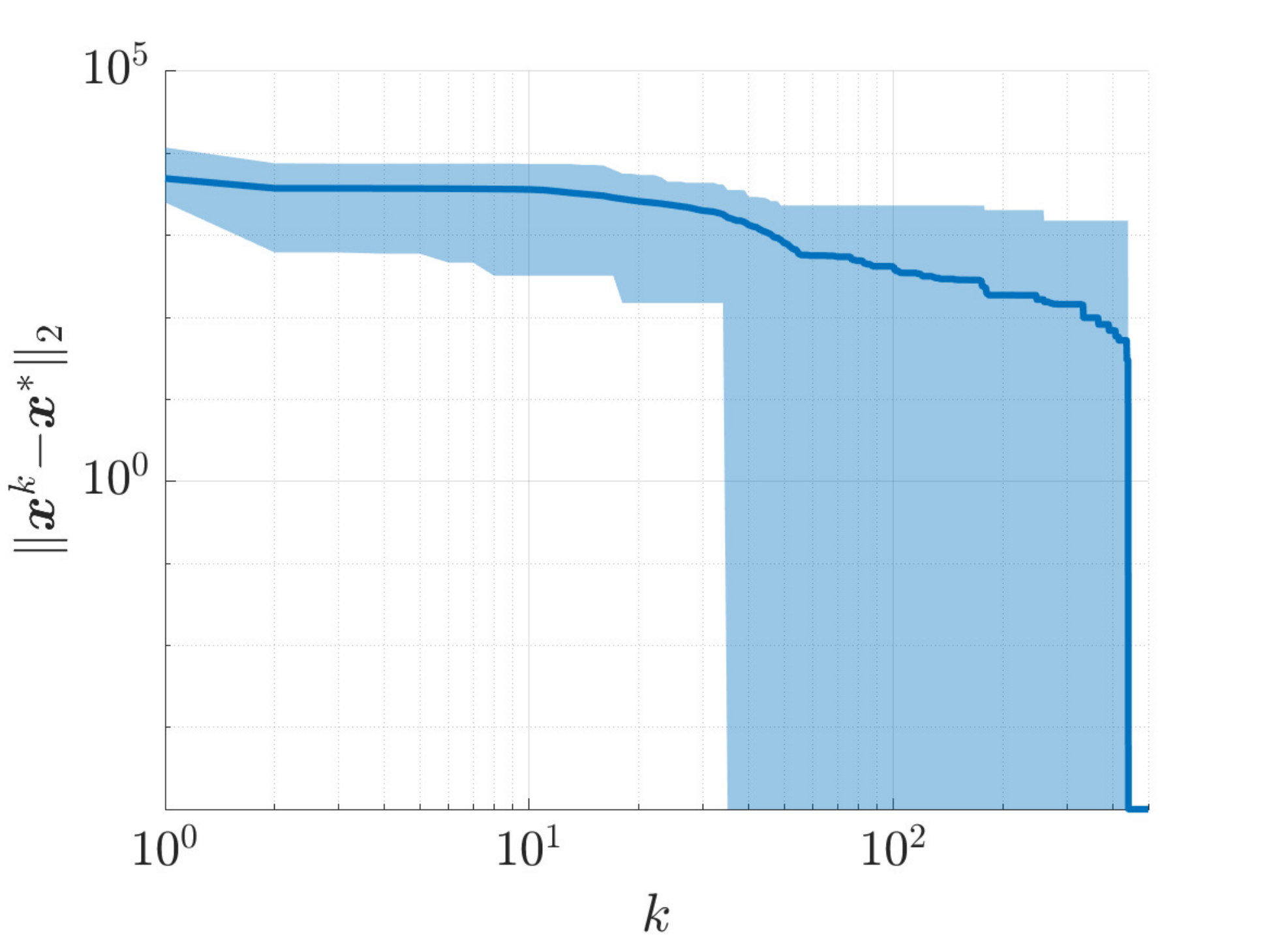}
		\caption{}
		\label{fig:MINE_distance_approx}
	\end{subfigure}
	\begin{subfigure}{.7\columnwidth}
		\centering
		\includegraphics[width=\textwidth]{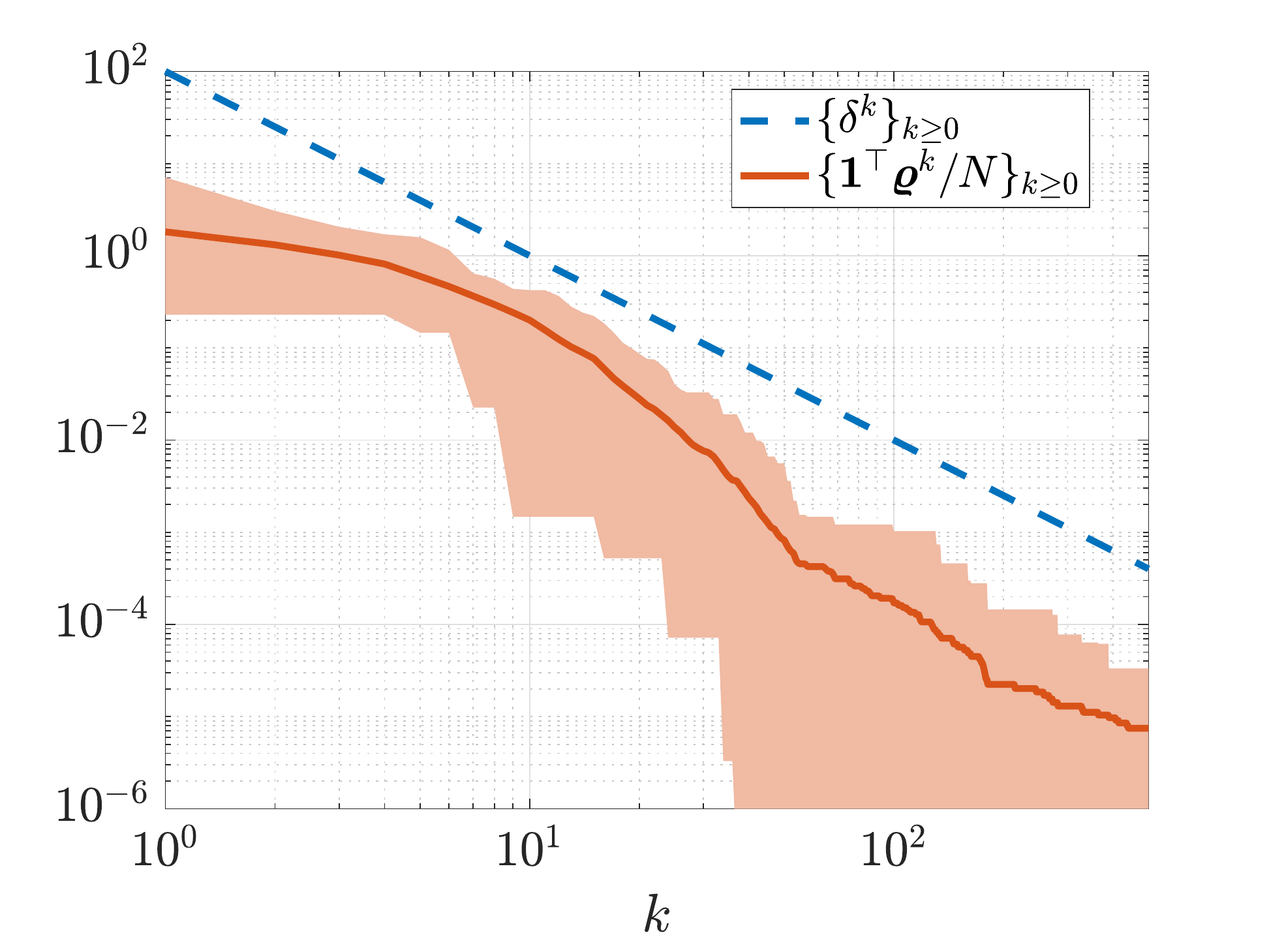}
		\caption{}
		\label{fig:approx_error}
	\end{subfigure}
	\caption{(a) Distance between the strategy profile $\bs{x}^k$ generated by Algorithm~\ref{alg:exact_ordinal} and an \gls{MI-NE}, averaged over the considered $50$ randomly generated instances of \eqref{eq:cournot}; (b) Distance between the strategy profile $\bs{x}^k$ generated by Algorithm~\ref{alg:inexact_exact} and an \gls{MI-NE} of the Cournot model in \eqref{eq:cournot}, averaged over the considered $50$ randomly generated instances. The sequence of strategy profiles $\{\bs{x}^k\}_{k \geq 0}$ converges to a $10^{-6}$-approximate \gls{MI-NE}; (c) Averaged evolution over $k$ of the approximation error measured by the solver in computing an inexact \gls{BR}, $\bs{\varrho}^k$ (red line) -- Algorithm~\ref{alg:inexact_exact}. This quantity is upper bounded by the error tolerance sequence $\{\delta^k\}_{k \geq 0}$ in Table~\ref{tab:sim_val} (blue line).}
	\label{fig:figures}
\end{figure*}

The numerical results reported in Figure~\ref{fig:figures} are obtained in Matlab by using Gurobi \cite{gurobi} as a solver on a laptop with a Quad-Core Intel Core i5 2.4\,GHz CPU and 8\,GB RAM. Specifically, we generate $50$ random instances of the considered \gls{MI-NEP} in \eqref{eq:cournot}, and test the behavior of Algorithm~\ref{alg:exact_ordinal} and \ref{alg:inexact_exact}, where each agent takes, on average, $0.0788$\,s to compute a \gls{BR} strategy.
While Fig.~\ref{fig:figures}(a) shows the averaged convergence behavior of the sequence of \gls{MI} strategy profiles generated by Algorithm~\ref{alg:exact_ordinal}, which actually converge in less than $30$ iterations,
Fig.~\ref{fig:figures}(b), instead, illustrates the averaged behavior of the sequence of sub-optimal \gls{MI} strategy profiles generated by Algorithm~\ref{alg:inexact_exact}, which converges to an $10^{-6}$-approximate \gls{MI-NE} of the Cournot model in \eqref{eq:cournot}. Note that the if-condition in the procedure generates a typical staircase behaviour and the convergence, in general, requires few more iterations. Finally, Fig.~\ref{fig:figures}(c) illustrates the averaged (over the agents) evolution of $\bs{\varrho}^k \coloneqq \col((\varrho_i^k)_{i \in \mc{I}})$, which corresponds to the approximation error actually made in computing an inexact \gls{BR}. Each $\varrho_i^k$, indeed, denotes the distance returned by the solver between the inexact and exact \gls{BR} solutions made by agent $i$ at the $k$-th iteration. The quantity $\bsone^\top \bs{\varrho}^k/N$ is always upper bounded by the error sequence $\{\delta^k\}_{k \geq 0}$ reported in Table~\ref{tab:sim_val}, as expected.

\section{Conclusion}
\label{sec:conclusion}

We have presented two proximal-like equilibrium seeking algorithms for \glspl{NEP} with mixed-integer variables admitting either generalized ordinal or exact potential functions. Exploiting the properties of integer-compatible regularization functions used as penalty terms in the agents' cost functions is key to prove convergence both in case the agents pursue an exact optimal strategy or an approximated one.

Future research directions include, but are not limited to, the extension of the proposed algorithms to generalized \glspl{MI-NEP} as well as developing their stochastic counterparts. There are also interesting computational questions to investigate. From the proximal-point interpretation of the feasibility pump, large penalty parameters enforce integer restrictions. 
One possibility to approach this within the proposed equilibrium seeking algorithms is to design a double loop procedure with large penalties at the beginning of the scheme.


\balance


\end{document}